\documentclass[11pt]{amsart}

\usepackage{amsmath,amssymb,amsthm,mathtools}
\usepackage{enumitem}
\usepackage{aliascnt}
\usepackage[hidelinks]{hyperref}
\usepackage[nameinlink,capitalise]{cleveref}
\hypersetup{
 pdftitle={Extremal Lin--Lu--Yau Curvature: Graph Density, Girth, and Short Cycles},
 pdfauthor={Qing Xia},
 pdfkeywords={Lin--Lu--Yau curvature, extremal curvature, weighted graph,
maximum average degree, girth, short cycles, locally finite graph}
}

\newtheorem{theorem}{Theorem}[section]
\newaliascnt{proposition}{theorem}
\newtheorem{proposition}[proposition]{Proposition}
\aliascntresetthe{proposition}
\newaliascnt{lemma}{theorem}
\newtheorem{lemma}[lemma]{Lemma}
\aliascntresetthe{lemma}
\newaliascnt{corollary}{theorem}

\aliascntresetthe{corollary}
\theoremstyle{definition}
\newaliascnt{definition}{theorem}
\newtheorem{definition}[definition]{Definition}
\aliascntresetthe{definition}
\newaliascnt{example}{theorem}
\newtheorem{example}[example]{Example}
\aliascntresetthe{example}
\theoremstyle{remark}
\newaliascnt{remark}{theorem}
\newtheorem{remark}[remark]{Remark}
\aliascntresetthe{remark}

\newcommand{\Curv}{\mathrm{Curv}}
\newcommand{\LLY}{\mathrm{LLY}}
\newcommand{\R}{\mathbb{R}}
\newcommand{\mad}{\operatorname{mad}}
\newcommand{\girth}{\operatorname{girth}}
\newcommand{\Kmax}{\kappa_{\max}^{\LLY}}

\title{Extremal Lin--Lu--Yau Curvature: Graph Density, Girth, and Short Cycles}
\author{Qing Xia}
\address{School of Mathematical Sciences\\
University of Science and Technology of China\\
96 Jinzhai Road\\
Hefei 230026, Anhui Province\\
China}
\email{xq0420@mail.ustc.edu.cn}
\date{\today}

\subjclass[2020]{05C12, 05C63, 05C81, 60J10}
\keywords{Lin--Lu--Yau curvature, extremal curvature, weighted graph,
maximum average degree, girth, short cycles, locally finite graph}

\begin{document}

\begin{abstract}
We develop a variational framework for extremal Lin--Lu--Yau
curvature under positive edge reweighting with the combinatorial
metric fixed. For a connected locally finite graph $G$, define
\[
\Kmax(G)
:=
\sup_{w>0}\inf_{e\in E(G)}\kappa_{\LLY}^w(e).
\]
This invariant relates edgewise discrete curvature to global
combinatorial and topological data. In the high-girth regime, we prove
for every connected locally finite graph the exact identity
\[
\Kmax(G)=\frac4{\mad(G)}-2,
\]
with the convention $1/\infty=0$. The
formula is also equivalent to a normalized Euler-characteristic
density:
\[
\Kmax(G)
=
\inf_{\substack{H\subseteq G\text{ finite}\\E(H)\ne\varnothing}}
\frac{2(c(H)-\beta(H))}{|E(H)|}.
\]
For arbitrary locally finite graphs, a nonnegative
short-cycle surplus yields girth-dependent density bounds, with
rigidity results in the finite case. We also classify the locally
finite graphs of girth at least $6$ admitting a positive edge weight
with nonnegative curvature on every edge: they are precisely the
trees and the finite cycles.
\end{abstract}

\maketitle

\section{Introduction and statement of results}

Using local curvature notions to investigate global combinatorial structure has become an active direction in discrete geometry and graph theory. In many classical structural and extremal problems, one works under strong global regularity or symmetry assumptions. Curvature conditions provide a different perspective: they impose local geometric constraints while allowing the underlying graph to be far from globally symmetric. This makes it possible to extend structural questions from highly regular graph classes to much broader families governed only by local curvature conditions. In the present paper, working with Ollivier-type discrete Ricci curvature and, more specifically, Lin--Lu--Yau curvature, we introduce a new variational invariant under positive edge reweighting and use it to connect local curvature with global combinatorial density, cycle structure, and topological information.

Ollivier's coarse Ricci curvature compares the Wasserstein distance between one-step probability distributions based at two points with the distance between the points themselves \cite{Ollivier2009}. Lin, Lu, and
Yau introduced the high-idleness graph curvature now bearing their
names \cite{LinLuYau2011}, and the dependence on the idleness
parameter was developed further in \cite{BCLMP2018}. M\"unch and
Wojciechowski extended the theory to general graph Laplacians and
obtained the limit-free Laplacian formulation used below
\cite{MunchWojciechowski2019}.

\begin{definition}[Extremal curvature under edge reweighting]
\label{def:general-extremal-curvature}
Let $\Curv$ be a discrete curvature notion on positively
edge-weighted locally finite graphs that is invariant under global
rescaling of the edge weights. Let $G$ be connected and locally finite. Let $\mathcal S_{\Curv}(G)$ be the set
on which the curvature is evaluated; we call
$\mathcal S_{\Curv}(G)$ the \emph{curvature carrier} of $\Curv$ on
$G$. Put
\[
\kappa_{\Curv}^w(G)
:=
\inf_{\xi\in\mathcal S_{\Curv}(G)}\kappa_{\Curv}^w(\xi)
\]
and define
\begin{equation}
\kappa_{\max}^{\Curv}(G)
:=
\sup_{w:E(G)\to(0,\infty)}\kappa_{\Curv}^w(G).
\end{equation}
We call $\kappa_{\max}^{\Curv}(G)$ the
\emph{extremal $\Curv$-curvature} of $G$.
\end{definition}

For Lin--Lu--Yau curvature,
\[
\mathcal S_{\LLY}(G)=E(G),
\]
whereas for a vertex-based Bakry--\'Emery curvature model,
\[
\mathcal S_{\mathrm{BE}}(G)=V(G).
\]
The present paper develops \cref{def:general-extremal-curvature} for
Lin--Lu--Yau curvature. We work in the fixed-combinatorial-distance
weighted model: the graph and its combinatorial metric are fixed,
while positive edge weights determine the nearest-neighbor transition probabilities. Accordingly, for a connected locally finite graph $G$ with
$E(G)\ne\varnothing$, we write
\begin{equation}
\label{eq:kmax-intro}
\kappa_{\LLY}^w(G)
=
\inf_{e\in E(G)}\kappa_{\LLY}^w(e),
\qquad
\Kmax(G)
=
\sup_{w>0}\kappa_{\LLY}^w(G).
\end{equation}
When $G$ is finite, the infimum is a minimum. Note that the supremum in
\eqref{eq:kmax-intro} need not be attained.

For a connected locally finite graph \(G\), define the \emph{maximum average
degree} through finite subgraphs by
\[
\mad(G)
:=
\sup_{\substack{H\subseteq G\text{ finite}\\ V(H)\ne\varnothing}}
\frac{2|E(H)|}{|V(H)|}.
\]
Equivalently, if \(F\subseteq E(G)\) is finite and nonempty and
\[
V(F):=\{x\in V(G):x\text{ is incident with some }e\in F\},
\]
then
\begin{equation}
\label{eq:mad-edge-set}
\mad(G)
=
\sup_{\substack{\varnothing\ne F\subseteq E(G)\\ F\text{ finite}}}
\frac{2|F|}{|V(F)|},
\qquad
\frac{2}{\mad(G)}
=
\inf_{\substack{\varnothing\ne F\subseteq E(G)\\ F\text{ finite}}}
\frac{|V(F)|}{|F|},
\end{equation}
with the convention \(1/\infty=0\). Indeed, every nonempty finite edge
set \(F\) determines the finite subgraph \((V(F),F)\) without isolated
vertices, while deleting isolated vertices from a finite subgraph does
not decrease its edge--vertex density. For finite graphs this is the usual maximum average degree. 

The
high-girth formula below shows that this global density determines the
extremal curvature without any additional hypothesis at infinity. We write $\girth(G)$ for the girth of $G$, with the convention
$\girth(G)=\infty$ when $G$ is a tree. For any finite graph $H$, we denote by $c(H)$ the number of connected
components of $H$, and define its cycle rank by
\[
\beta(H):=|E(H)|-|V(H)|+c(H).
\]

\begin{theorem}[High-girth density formula]
\label{thm:intro-high-girth}
Let $G$ be connected and locally finite with $E(G)\ne\varnothing$ and
$\girth(G)\ge6$. Then
\begin{equation}
\label{eq:intro-high-girth}
\Kmax(G)=\frac4{\mad(G)}-2,
\end{equation}
with the convention $1/\infty=0$. Equivalently,
\begin{equation}
\label{eq:intro-euler}
\Kmax(G)
=
\inf_{\substack{H\subseteq G\text{ finite}\\E(H)\ne\varnothing}}
\frac{2(c(H)-\beta(H))}{|E(H)|}.
\end{equation}
If $G$ is finite, the infimum in \eqref{eq:intro-euler} is a minimum.
\end{theorem}

For finite connected graphs of girth at least $6$, the sign of
$\Kmax(G)$ distinguishes trees, unicyclic graphs, and graphs with at
least two independent cycles:
\[
\begin{aligned}
\Kmax(G)>0
&\iff G\text{ is a tree},\\
\Kmax(G)=0
&\iff \beta(G)=1,\\
\Kmax(G)<0
&\iff \beta(G)\ge2.
\end{aligned}
\]
In particular, every finite
tree $T$ satisfies
\begin{equation}
\label{eq:intro-tree}
\Kmax(T)=\frac{2}{|E(T)|}.
\end{equation}

For finite graphs, attainment is characterized by strict
balancedness. Recall that $G$ is \emph{strictly balanced} if
\[
\frac{|E(H)|}{|V(H)|}
<
\frac{|E(G)|}{|V(G)|}
\qquad
\text{for every nonempty proper subgraph }H\subsetneq G.
\]
This is classical terminology from random graph theory; see, for
example, \cite[Section~4.1]{Bollobas2001}. This condition also appears in the prescribed-curvature problem of
Lin and Liu. They proved that, for a finite graph of girth at least
$6$, a positive constant-curvature weight exists exactly when the
graph is strictly balanced, and that such a weight is unique up to
global scaling \cite[Theorem~3.1 and Proposition~3.2]{LinLiu2026}.
Their result is an existence and uniqueness theorem for the
constant-curvature equation. Our variational problem is different: we
optimize the minimum Lin--Lu--Yau curvature over all positive edge
weights. In the high-girth setting, we determine the extremal value for every connected locally finite graph and, in the finite case, characterize exactly when it is attained and describe the limiting behavior of positive edge weights approaching the extremal value when the supremum is not attained.

A positive edge weight $w$ satisfying
\[
\kappa_{\LLY}^w(G)=\Kmax(G)
\]
is called a \emph{maximizing weight}.

\begin{theorem}[Attainment and uniqueness]
\label{thm:intro-attainment}
Let $G$ be finite, connected, and of girth at least $6$. Then
$\Kmax(G)$ is attained by a positive weight if and only if $G$ is
strictly balanced. In that case every maximizing weight is unique up
to global scaling and has constant Lin--Lu--Yau curvature equal to
$\Kmax(G)$.
\end{theorem}

Short cycles lead to a four-level density hierarchy.

\begin{theorem}[Girth-density hierarchy]
\label{thm:intro-girth-density}
Let $G$ be connected and locally finite, let $E(G)\ne\varnothing$,
and let $\ell\in\{3,4,5,6\}$. If $\girth(G)\ge\ell$, then
\begin{equation}
\label{eq:intro-girth-density}
\Kmax(G)
\le
4-\ell+\frac{\ell-2}{\mad(G)}.
\end{equation}
If $G$ is infinite, then in addition
\begin{equation}
\label{eq:intro-infinite-girth-density}
\Kmax(G)
\le
\min\left\{
0,
4-\ell+\frac{\ell-2}{\mad(G)}
\right\}.
\end{equation}
For $\ell=6$, the corresponding bounds are equalities.
\end{theorem}

The additional upper bound $0$ in the infinite case follows from the
weighted discrete Bonnet--Myers theorem
\cite[Proposition~1.2]{Xia2026CDNP}; see also the general weighted
Laplacian framework of M\"unch and Wojciechowski
\cite{MunchWojciechowski2019}.

For infinite graphs, the additional upper bound $0$ makes the
equality problem for \eqref{eq:intro-girth-density} degenerate in part.
When $\ell=3$, the right-hand side of
\eqref{eq:intro-girth-density} is strictly positive. When $\ell=4$,
it equals $2/\mad(G)$, which is strictly positive if
$\mad(G)<\infty$. Hence equality is impossible in these cases. The same conclusion holds for $\ell=5$ whenever $\mad(G)<3$. The remaining infinite equality problem at girth five, when $\mad(G)\ge3$, is of a different nature and will not be pursued here. We therefore focus here on the finite equality cases, for which the local estimates yield the following sharp rigidity results. 

\begin{theorem}[Rigidity at girth three]
\label{thm:girth3-rigidity}
Let $G$ be finite and connected with $E(G)\ne\varnothing$. The bound
\[
\Kmax(G)\le1+\frac1{\mad(G)}
\]
is attained by a positive weight if and only if $G$ is the complete graph $K_n$ ($n\ge2$). In
particular,
\[
\Kmax(K_n)=\frac{n}{n-1},
\]
and every maximizing weight on $K_n$ is constant up to global
scaling.
\end{theorem}

For a vertex $x\in V(G)$, let $N(x)$ be the set of neighbors of $x$. For an edge $e\in
E(G)$, define its \emph{local girth} by
\[
\girth_G(e)
:=
\min\{|E(C)|:C\text{ is a cycle of }G\text{ containing }e\},
\]
with $\girth_G(e)=\infty$ if $e$ lies on no cycle. Hence $\girth(G)=\inf_{e\in E(G)}\girth_G(e)$. For $\ell\in\{4,5\}$ and an edge $e=xy$ with
$\girth_G(e)\ge\ell$, let $\mathcal M_\ell(e)$ be the bipartite
graph with parts
\[
N(x)\setminus\{y\},
\qquad
N(y)\setminus\{x\},
\]
in which $u$ and $v$ are adjacent exactly when their combinatorial distance is $\ell-3$.

\begin{theorem}[Rigidity at girth four and five]
\label{thm:girth45-rigidity}
Let $\ell\in\{4,5\}$ and let $G$ be finite and connected with
$\girth(G)\ge\ell$. Then the bound
\[
\Kmax(G)
\le
4-\ell+\frac{\ell-2}{\mad(G)}
\]
is attained by a positive weight if and only if $G$ is $d$-regular
for some $d$ and $\mathcal M_\ell(e)$ has a perfect matching for
every $e\in E(G)$. In that case
\[
\Kmax(G)=4-\ell+\frac{\ell-2}{d},
\]
and every maximizing weight is constant up to global scaling.
\end{theorem}

The same fixed-combinatorial-distance weighted model has recently
been studied on infinite graphs by Hua, Lin, and Liu
\cite{HuaLinLiu2026}. Their work is dynamical: they study the
existence, uniqueness, and convergence of a Ricci flow toward
prescribed Lin--Lu--Yau curvature. Our problem is instead
variational, asking how large the minimum curvature can be among all
positive edge weights.

At the nonnegative-curvature threshold, the high-girth setting
admits a complete locally finite classification.

\begin{theorem}[Nonnegative-curvature classification]
\label{thm:intro-nonnegative-classification}
Let $G$ be connected and locally finite with $E(G)\ne\varnothing$ and
$\girth(G)\ge6$. Then there exists a positive edge weight $w$ such
that
\[
\kappa_{\LLY}^w(e)\ge0
\qquad(e\in E(G))
\]
if and only if $G$ is a tree or a finite cycle $C_n$ with $n\ge6$. Every locally finite
tree admits a weight with strictly positive curvature on every edge;
in the infinite case the curvature infimum is necessarily zero. On $C_n$, a positive edge weight has nonnegative curvature on every
edge if and only if it is constant up to global scaling; in that case
the curvature is identically zero.
\end{theorem}

The existence statement in the finite case is closely related to the
curvature-distortion theory developed in \cite{Xia2026CDN}. There,
one additionally requires the positive edge weight to have finite
distortion
\[
\operatorname{Dist}(w)
:=
\frac{\sup_{e\in E(G)}w_e}
     {\inf_{e\in E(G)}w_e}
<\infty.
\]
For a finite graph every positive edge weight automatically has
finite distortion, so the finite part of the preceding classification
recovers the corresponding existence result
of \cite{Xia2026CDN}. For infinite graphs, however, finite distortion
is an additional global restriction and the two problems are no
longer equivalent. The present theorem imposes no uniform control on
the edge weights; in particular, it shows that every locally finite
tree admits a positive edge weight with strictly positive
Lin--Lu--Yau curvature on every edge. Thus, in the infinite setting, the present
classification is genuinely different from the finite-distortion
results of \cite{Xia2026CDN}.

The paper is organized as follows. Section~\ref{sec:prelim} develops
the weighted Lin--Lu--Yau model and the basic locally finite
identities. Section~\ref{sec:density} proves
\cref{thm:intro-high-girth,thm:intro-attainment}. Its main purpose is to determine the extremal curvature in the high-girth setting and to understand when the extremal value can actually be attained.
Section~\ref{sec:short-cycle} proves
\cref{thm:intro-girth-density} and develops the local estimates needed to understand how triangles, quadrilaterals, and pentagons affect the extremal curvature. Section~\ref{sec:equality} proves
\cref{thm:girth3-rigidity,thm:girth45-rigidity}, thereby characterizing the finite graphs for which the corresponding girth-dependent upper bounds are attained. Finally, Section~\ref{sec:infinite} proves
\cref{thm:intro-nonnegative-classification}.

\section{Weighted Lin--Lu--Yau curvature and basic identities}
\label{sec:prelim}

Throughout the paper, graphs are simple, connected, locally finite,
and satisfy $E(G)\ne\varnothing$ unless stated otherwise. We write
$G=(V,E)$. For $x\in V(G)$, recall that $N(x)$ denotes the set of neighbors of $x$, define
\[
N[x]:=N(x)\cup\{x\},
\qquad
d_x:=|N(x)|,
\]
and let $d$ denote the combinatorial distance.

A positive edge weight is a map
\[
w:E(G)\to(0,\infty).
\]
Throughout, $w>0$ means that $w$ is a positive edge weight. For
$e=xy$ we use $w_e$ and $w_{xy}$ interchangeably. Local finiteness
ensures that the weighted degree
\[
d_x^w:=\sum_{z\sim x}w_{xz}
\]
is finite.

For $\alpha\in[0,1]$, define
\begin{equation}
\mu_x^{\alpha,w}(z)=
\begin{cases}
\alpha, & z=x,\\[2mm]
(1-\alpha)\dfrac{w_{xz}}{d_x^w}, & z\sim x,\\[3mm]
0, & \text{otherwise}.
\end{cases}
\end{equation}
For probability measures $\mu,\nu$ on $V(G)$, let $W_1(\mu,\nu)$ be
the $1$-Wasserstein distance with respect to $d$. For adjacent
vertices $x\sim y$, put
\[
\kappa_\alpha^w(x,y)
:=
1-W_1(\mu_x^{\alpha,w},\mu_y^{\alpha,w})
\]
and
\begin{equation}
\kappa_{\LLY}^w(x,y)
:=
\lim_{\alpha\to1^-}
\frac{\kappa_\alpha^w(x,y)}{1-\alpha}.
\end{equation}
This is the fixed-combinatorial-distance weighted model used below.
The basic idleness theory originates in
\cite{LinLuYau2011,BCLMP2018}.

For a function $f:V(G)\to\R$, write
\[
\operatorname{Lip}(1)
:=
\{f:|f(u)-f(v)|\le d(u,v)\text{ for all }u,v\in V(G)\}
\]
and
\[
\Delta_wf(x)
:=
\frac1{d_x^w}
\sum_{z\sim x}w_{xz}\bigl(f(z)-f(x)\bigr).
\]
We use the limit-free formulation of M\"unch and Wojciechowski
\cite{MunchWojciechowski2019}: for every adjacent pair $x\sim y$,
\begin{equation}
\label{eq:dual-LLY}
\kappa_{\LLY}^w(x,y)
=
\inf\left\{
\Delta_wf(x)-\Delta_wf(y):
 f\in\operatorname{Lip}(1),\ f(y)-f(x)=1
\right\}.
\end{equation}
For a fixed edge $xy$, the objective depends only on the finite set
$N[x]\cup N[y]$. After normalizing $f(x)=0$, the admissible
restrictions to this set form a compact polytope. Moreover, every
such restriction admits a compactly supported $1$-Lipschitz
extension to $V(G)$: prescribe the value $0$ outside a sufficiently
large finite ball and apply the McShane extension theorem
\cite{McShane1934}. Hence the infimum in \eqref{eq:dual-LLY} may be
taken over these local restrictions and is attained also in the
locally finite setting.

The curvature is invariant under global rescaling of the edge
weights:
\[
\kappa_{\LLY}^{cw}(e)=\kappa_{\LLY}^w(e)
\qquad(c>0,\ e\in E(G)).
\]
When $w\equiv1$, we suppress the superscript $w$.

\begin{definition}[Maximizing sequence]
Recall that a positive edge weight attaining $\Kmax(G)$ is called a
\emph{maximizing weight}.
A sequence of positive edge weights
$\{w^{(j)}\}_{j\ge1}$ is called a \emph{maximizing sequence} if
\[
\kappa_{\LLY}^{w^{(j)}}(G)
\longrightarrow
\Kmax(G).
\]
\end{definition}

For $e=xy$, put
\begin{equation}
q_e^w
:=
\frac{w_e}{d_x^w}
+
\frac{w_e}{d_y^w}.
\end{equation}
We also use the notation $q_{xy}^w$ when the endpoints are more
convenient than the edge label.

For a finite edge set $F\subseteq E(G)$, recall the notation
$V(F)$ introduced in the Introduction, and put
\[
G_F:=(V(F),F).
\]
For a finite nonempty subgraph $H\subseteq G$, define its normalized
$w$-\emph{leakage} by
\begin{equation}
Leak_w(H)
:=
\sum_{x\in V(H)}
\sum_{\substack{e\ni x\\e\notin E(H)}}
\frac{w_e}{d_x^w}.
\end{equation}
The sum is finite by local finiteness. Notice that $Leak_w(H)$ depends
on the subgraph itself, not only on its vertex set.

Define
\begin{equation}
\label{eq:Q}
\mathcal Q(G)
:=
\sup_{w>0}\inf_{e\in E(G)}q_e^w.
\end{equation}
Note that for finite graphs the infima in \eqref{eq:Q} are minima.

\begin{lemma}[Finite-edge-set leakage identity]
\label{lem:edge-set-q}
Let $F\subseteq E(G)$ be finite and nonempty. Then
\begin{equation}
\sum_{e\in F}q_e^w
=
|V(F)|-Leak_w(G_F)
\le
|V(F)|.
\end{equation}
Equality in the last inequality holds if and only if $F=E(G)$. In
particular, the inequality is strict whenever $G$ is infinite.
\end{lemma}

\begin{proof}
For each incidence $x\in e$, write
\[
p_{x,e}:=\frac{w_e}{d_x^w}.
\]
Then
\[
\sum_{e\in F}q_e^w
=
\sum_{x\in V(F)}
\sum_{\substack{e\in F\\e\ni x}}p_{x,e}
=
\sum_{x\in V(F)}
\left(
1-
\sum_{\substack{e\ni x\\e\notin F}}p_{x,e}
\right)
=
|V(F)|-Leak_w(G_F).
\]
Since all weights are positive, $Leak_w(G_F)=0$ exactly when no edge
of $E(G)\setminus F$ is incident with a vertex of $V(F)$. By
connectedness and $F\ne\varnothing$, this is equivalent to $F=E(G)$.
\end{proof}

\begin{lemma}[Basic curvature estimate]
\label{lem:baseline-curvature}
For every edge $e\in E(G)$,
\begin{equation}
\label{eq:baseline-curvature}
\kappa_{\LLY}^w(e)\ge2q_e^w-2.
\end{equation}
Moreover,
\begin{equation}
\label{eq:local-high-girth}
\kappa_{\LLY}^w(e)=2q_e^w-2
\quad\Longleftrightarrow\quad
\girth_G(e)\ge6.
\end{equation}
\end{lemma}

\begin{proof}
Write $e=xy$ and normalize an admissible test function in
\eqref{eq:dual-LLY} by $f(x)=0$ and $f(y)=1$. For every
$u\sim x$, $u\ne y$, one has $f(u)\ge-1$, while for every
$v\sim y$, $v\ne x$, one has $f(v)\le2$. Hence
\begin{equation}
\label{eq:termwise-lower-bounds}
\Delta_wf(x)-\Delta_wf(y)
\ge
2q_e^w-2.
\end{equation}
Taking the infimum proves \eqref{eq:baseline-curvature}.

If $\girth_G(e)\ge6$, prescribe
\[
f(x)=0,
\quad
f(y)=1,
\quad
f=-1\text{ on }N(x)\setminus\{y\},
\quad
f=2\text{ on }N(y)\setminus\{x\}.
\]
The two neighbor sets are disjoint. If
$u\in N(x)\setminus\{y\}$ and $v\in N(y)\setminus\{x\}$, then a
$u$--$v$ path of length at most $2$ would create a cycle through $e$
of length at most $5$. Thus $d(u,v)=3$, and the prescription is
$1$-Lipschitz. Its McShane extension gives equality in
\eqref{eq:baseline-curvature}.

Conversely, suppose equality holds. Since the infimum in
\eqref{eq:dual-LLY} is attained, equality must hold term by term in
\eqref{eq:termwise-lower-bounds}. A common neighbor of $x$ and $y$
would then have to take both values $-1$ and $2$. If $e$ lies on a
$4$- or $5$-cycle but on no triangle, there exist
$u\in N(x)\setminus\{y\}$ and $v\in N(y)\setminus\{x\}$ with
$d(u,v)\le2$, while equality would force $f(u)=-1$ and $f(v)=2$.
Both cases contradict the Lipschitz condition. Hence
$\girth_G(e)\ge6$.
\end{proof}

Consequently, if $\girth(G)\ge6$, then
\begin{equation}
\label{eq:kmax-Q}
\Kmax(G)=2\mathcal Q(G)-2.
\end{equation}

\section{Exact density optimization and the high-girth formula}
\label{sec:density}

For a locally finite graph $G$ and a vertex set $S\subseteq V(G)$, define
\[
E_G(S):=\{xy\in E(G):x,y\in S\}.
\]
When the ambient graph $G$ is clear, we simply write $E(S)$.

We first prove the upper density bound.

\begin{proposition}[Density upper bound]
\label{prop:density-upper}
For every connected locally finite graph $G$ with $E(G)\ne\varnothing$,
\begin{equation}
\label{eq:density-upper}
\mathcal Q(G)
\le
\frac2{\mad(G)}.
\end{equation}
\end{proposition}

\begin{proof}
For every finite nonempty $F\subseteq E(G)$, \cref{lem:edge-set-q} gives
\[
|F|\inf_{e\in E(G)}q_e^w
\le
\sum_{e\in F}q_e^w
\le
|V(F)|.
\]
Hence
\[
\inf_{e\in E(G)}q_e^w
\le
\inf_{\substack{\varnothing\ne F\subseteq E(G)\\ F\text{ finite}}}
\frac{|V(F)|}{|F|}.
\]
By \eqref{eq:mad-edge-set},
\[
\inf_{\substack{\varnothing\ne F\subseteq E(G)\\ F\text{ finite}}}
\frac{|V(F)|}{|F|}
=
\frac2{\mad(G)}.
\]
Taking the supremum over $w$ proves \eqref{eq:density-upper}.
\end{proof}

The reverse inequality is obtained from a nonlinear balancing
problem. Put
\[
\sigma(t):=\frac{e^t}{1+e^t}.
\]
For a potential $u:V(G)\to\R$, define the \emph{logit load of $u$ at $x$} by
\[
L_x^G(u)
:=
\sum_{y\sim x}\sigma(u_y-u_x).
\]
Local finiteness makes the sum finite.

\subsection{A uniform logit multiplier}

We first prove a finite existence lemma for the potentials used below.

\begin{lemma}[Finite strict logit feasibility]
\label{lem:finite-logit}
Let $H=(W,F)$ be a finite graph, and let
$c:W\to(0,\infty)$. For $S\subseteq W$, put
\[
c(S):=\sum_{x\in S}c_x,
\qquad
E_H(S):=\{xy\in F:x,y\in S\}.
\]
If every nonempty vertex set $S\subseteq W$ satisfies
\begin{equation}
\label{eq:finite-strict-Hall}
|E_H(S)|<c(S),
\end{equation}
then there exists $u:W\to [0,\infty)$ such that
\[
L_x^H(u)
=
\sum_{y\sim x}\sigma(u_y-u_x)
\le c_x
\qquad(x\in W).
\]
\end{lemma}

\begin{proof}
Let $u:W\to [0,\infty)$. Consider
\[
\mathcal J(u)
:=
\sum_{xy\in F}\log\bigl(e^{-u_x}+e^{-u_y}\bigr)
+
\sum_{x\in W}c_xu_x.
\]
For every $x\in W$,
\[
\frac{\partial\mathcal J}{\partial u_x}(u)
=
c_x-L_x^H(u).
\]
Since $W$ is finite and \eqref{eq:finite-strict-Hall} is strict, there
is $\gamma>0$ such that
\[
c(S)-|E_H(S)|
\ge
\gamma|S|
\qquad
(\varnothing\ne S\subseteq W).
\]
Minimize $\mathcal J$ on the closed set
\[
\mathcal O:=\{u\in\R^W:u_x\ge0\text{ for all }x\in W\}.
\]
For $u\in\mathcal O$,
\[
\log(e^{-u_x}+e^{-u_y})
\ge
-\min\{u_x,u_y\},
\]
and set
\[
S_t:=\{x\in W:u_x>t\}.
\]
Writing $\mathbf 1_A$ for the indicator of a set $A$, we have
\[
u_x=\int_0^\infty \mathbf 1_{\{u_x>t\}}\,dt
\]
and
\[
\min\{u_x,u_y\}
=
\int_0^\infty
\mathbf 1_{\{u_x>t,\;u_y>t\}}\,dt.
\]
Hence
\[
\sum_{x\in W}c_xu_x
=\int_0^\infty \sum_{x\in W}c_x\mathbf 1_{\{u_x>t\}}\,dt=
\int_0^\infty c(S_t)\,dt
\]
and
\[
\sum_{xy\in F}\min\{u_x,u_y\}
=\int_0^\infty
\sum_{xy\in F}\mathbf 1_{\{u_x>t,\;u_y>t\}}\,dt=
\int_0^\infty |E_H(S_t)|\,dt.
\]
Therefore
\[
\mathcal J(u)
\ge
\int_0^\infty
\bigl(c(S_t)-|E_H(S_t)|\bigr)\,dt
\ge
\gamma\int_0^\infty |S_t|\,dt
=
\gamma\sum_{x\in W}u_x.
\]
Thus
\[
\mathcal J(u)\to\infty
\qquad\text{as}\qquad
\sum_{x\in W}u_x\to\infty
\quad (u\in\mathcal O).
\]
Since $\mathcal O$ is closed and $\mathcal J$ is continuous, $\mathcal J$ attains its minimum on $\mathcal O$ at some $u^\ast\in\mathcal O$. We now use the standard first-order condition at the minimizer $u^\ast$. If $u_x^\ast>0$, the first-order condition gives
$L_x^H(u^\ast)=c_x$. If $u_x^\ast=0$, the one-sided derivative in the
positive coordinate direction is nonnegative, so
$L_x^H(u^\ast)\le c_x$.
\end{proof}

Let $G$ be a connected locally finite graph. A vertex set $P\subseteq V(G)$ is called connected if $G[P]$ is connected where $G[P]$ denotes the vertex-induced subgraph of $P$. For disjoint vertex sets $A,B\subseteq V(G)$, write $E(A,B)$ for the
set of edges with one endpoint in $A$ and one in $B$. To extend a potential while keeping its values fixed on a finite vertex set, we need a density bound that also accounts for edges entering the fixed vertex set. The following lemma produces a finite enlargement with this property.

\begin{lemma}[Finite boundary-controlled enlargement]
\label{lem:Hall-closure}
Let $G$ be a connected locally finite graph, and let $K,\delta>0$. Assume
\begin{equation}
\label{eq:uniform-density-gap}
|E(S)|
\le
(K-\delta)|S|
\qquad
(\varnothing\ne S\subseteq V(G),\ S\text{ finite}),
\end{equation}
where $E(S)$ denotes the set of edges with both endpoints in $S$.
Then every finite connected set $P_0\subseteq V(G)$ is contained in a
finite connected set $P^\sharp$ such that for every nonempty finite set $T\subseteq V(G)\setminus P^\sharp$
\begin{equation}
\label{eq:Hall-closed}
|E(T)|+|E(T,P^\sharp)|
\le
\left(K-\frac\delta2\right)|T|.
\end{equation}
\end{lemma}

\begin{proof}
Starting from $P_0$, we construct finite connected sets
\[
P_0\subseteq P_1\subseteq P_2\subseteq\cdots
\]
as follows. Suppose that $P_j$ ($j\ge0$) has been constructed. If
\eqref{eq:Hall-closed} holds with $P^\sharp=P_j$, then we stop and set
\[
P^\sharp:=P_j.
\]
Otherwise, there exists a nonempty finite set
\[
S_j\subseteq V(G)\setminus P_j
\]
such that
\[
|E(S_j)|+|E(S_j,P_j)|
>
\left(K-\frac\delta2\right)|S_j|.
\]

Let $c(G[S_j])=r_j$, and let $S_j^1,\ldots,S_j^{r_j}$ be the vertex sets of the connected
components of $G[S_j]$. Then
\[
|E(S_j)|
=
\sum_{i=1}^{r_j}|E(S_j^i)|,
\qquad
|E(S_j,P_j)|
=
\sum_{i=1}^{r_j}|E(S_j^i,P_j)|,\qquad
|S_j|
=
\sum_{i=1}^{r_j}|S_j^i|.
\]
Hence, for some $i$, the set $S_j^i$ satisfies
\[
|E(S_j^i)|+|E(S_j^i,P_j)|
>
\left(K-\frac\delta2\right)|S_j^i|.
\]
Denote this set by $T_j$.

We claim that
\[
E(T_j,P_j)\ne\varnothing.
\]
Indeed, if $E(T_j,P_j)=\varnothing$, then
\[
|E(T_j)|
>
\left(K-\frac\delta2\right)|T_j|
>
(K-\delta)|T_j|,
\]
contradicting \eqref{eq:uniform-density-gap}. Therefore
\[
P_{j+1}:=P_j\cup T_j
\]
is finite and connected.

Suppose that the construction has continued through $P_m$ ($m\ge1$). Then, for each
$j=0,\ldots,m-1$, the set $P_j$ does not satisfy \eqref{eq:Hall-closed}, and hence
\[
|E(P_{j+1})|-|E(P_j)|
=
|E(T_j)|+|E(T_j,P_j)|
>
\left(K-\frac\delta2\right)|T_j|.
\]
Let
\[
A_m:=P_m\setminus P_0\qquad (m\ge1).
\]
Since the sets $T_0,\ldots,T_{m-1}$ are pairwise disjoint and
\[
A_m=\bigsqcup_{j=0}^{m-1}T_j,
\]
summing the preceding inequalities gives
\[
|E(P_m)|
>
|E(P_0)|
+
\left(K-\frac\delta2\right)|A_m|.
\]
On the other hand, applying \eqref{eq:uniform-density-gap} to $P_m$
gives
\[
|E(P_m)|
\le
(K-\delta)|P_m|
=
(K-\delta)\bigl(|P_0|+|A_m|\bigr).
\]
Consequently,
\[
\frac\delta2|A_m|
<
(K-\delta)|P_0|-|E(P_0)|.
\]
Thus $|A_m|$ is bounded independently of $m$. Since every $T_j$ is
nonempty,
\[
|A_m|
=
\sum_{j=0}^{m-1}|T_j|
\ge m.
\]
Hence the construction must stop after finitely many steps. At the
terminal step, the corresponding set $P^\sharp$ satisfies
\eqref{eq:Hall-closed}, as required.
\end{proof}

\begin{lemma}[Frozen extension]
\label{lem:frozen-extension}
Let $G$ be a connected locally finite graph, and let $K,\delta>0$. Assume \eqref{eq:uniform-density-gap}, let $P\subseteq V(G)$ be finite and satisfy
\eqref{eq:Hall-closed}, and let $H=G[W]$ be a finite connected induced
subgraph with $P\subseteq W$. Given a potential $u^P:P\to\R$ and
numbers $\varepsilon_z>0$ for $z\in P$, there is a potential
$u:W\to\R$ such that
\[
u|_P=u^P,
\]
\[
L_x^H(u)
\le
K-\frac\delta4
\qquad(x\in W\setminus P),
\]
and
\begin{equation}
\label{eq:old-load}
L_z^H(u)-L_z^{H[P]}(u^P)
\le
\varepsilon_z
\qquad(z\in P).
\end{equation}
\end{lemma}

\begin{proof}
Put $R:=W\setminus P$. For $x\in R$, let
\[
b_P(x):=|E(\{x\},P)|,
\qquad
c_x:=K-\frac\delta4-b_P(x).
\]
Applying \eqref{eq:Hall-closed} to $\{x\}$ gives
$c_x\ge\delta/4>0$.

Let $C$ be a connected component of $H[R]$. For every nonempty
$T\subseteq C$, \eqref{eq:Hall-closed} gives
\[
|E_C(T)|+|E(T,P)|
\le
\left(K-\frac\delta2\right)|T|.
\]
Since
\[
\sum_{x\in T}b_P(x)=|E(T,P)|,
\]
we obtain
\[
|E_C(T)|
\le\sum_{x\in T}\big(K-\frac\delta2-b_P(x)\big)=
\sum_{x\in T}c_x-\frac\delta4|T|
<
\sum_{x\in T}c_x.
\]
By \cref{lem:finite-logit}, there is a potential $v^C:C\to\R$ with
\[
L_x^C(v^C)\le c_x
\qquad(x\in C).
\]

For each connected component $C$ of $H[R]$, choose a constant
$M_C>0$ and define
\[
u_x:=v_x^C-M_C
\qquad(x\in C),
\]
while setting
\[
u_z:=u_z^P
\qquad(z\in P).
\]
The shifts by $M_C$ do not change potential differences within $C$,
and hence
\[
L_x^C(u)=L_x^C(v^C)
\qquad(x\in C).
\]
Let $x\in R$, and let $C$ be the connected component of $H[R]$
containing $x$. Every boundary edge from $x\in C$ to $z\in P$ contributes
less than $1$ to the logit load of $u$ at $x$, so
\[
L_x^H(u)
\le
L_x^C(v^C)+b_P(x)
\le
K-\frac\delta4.
\]
For a boundary edge $xz$ with $x\in C$ and $z\in P$, its contribution to $L_z^H(u)$ is
\[
\sigma(u_x-u_z),
\]
which tends to zero as $M_C\to\infty$. Since $H[R]$ has finitely many connected components and finitely many boundary edges, the constants $M_C$ can be chosen so that the increase in the logit load of $u$ at each $z\in P$ satisfies
\[
L_z^H(u)-L_z^{H[P]}(u^P)
\le
\varepsilon_z.
\] This gives
\eqref{eq:old-load} while leaving $u^P$ unchanged.
\end{proof}

\begin{theorem}[Uniform logit multiplier]
\label{thm:uniform-logit}
Let $G$ be connected and locally finite with
$\mad(G)<\infty$, and let $K>0$. If
\[
K>\frac{\mad(G)}2,
\]
then there exists a potential $u:V(G)\to\R$ such that
\[
L_x^G(u)\le K
\qquad(x\in V(G)).
\]
\end{theorem}

\begin{proof}
Put
\[
\delta:=K-\frac{\mad(G)}2>0.
\]
Then every nonempty finite set $S\subseteq V(G)$ satisfies
\[
|E(S)|
\le
\frac{\mad(G)}2|S|
=
(K-\delta)|S|.
\]

If $G$ is finite, then for every nonempty $S\subseteq V(G)$,
\[
|E(S)|
\le
(K-\delta)|S|
<
K|S|.
\]
Hence \cref{lem:finite-logit}, applied with $c_x\equiv K$, gives a
potential $u:V(G)\to\R$ such that
\[
L_x^G(u)\le K
\qquad(x\in V(G)).
\]

Assume now that $G$ is infinite. Since $G$ is connected and locally
finite, $V(G)$ is countable. Write
\[
V(G)=\{x_1,x_2,\ldots\}.
\]
We first construct an increasing sequence of finite connected sets
\[
P_1\subseteq P_2\subseteq\cdots
\]
such that each $P_n$ satisfies \eqref{eq:Hall-closed} and
\[
\bigcup_{n\ge1}P_n=V(G).
\]

Apply \cref{lem:Hall-closure} to the set $\{x_1\}$ and denote the
resulting set by $P_1$. Suppose that $P_n$ ($n\ge1$) has been constructed.
Choose a finite path joining $x_{n+1}$ to $P_n$, and let $Q_{n+1}$ be
the union of $P_n$ with the vertices of this path. Then $Q_{n+1}$ is
finite and connected. Applying \cref{lem:Hall-closure} to
$Q_{n+1}$ gives a finite connected set $P_{n+1}$ such that
\[
Q_{n+1}\subseteq P_{n+1}
\]
and $P_{n+1}$ satisfies \eqref{eq:Hall-closed}. In particular,
$x_n\in P_n$ for every $n\ge1$, and hence
\[
\bigcup_{n\ge1}P_n=V(G).
\]

We next construct potentials
\[
u_n:P_n\to\R
\]
recursively so that
\[
u_{n+1}|_{P_n}=u_n.
\]

For $P_1$, define $c:P_1\to(0,\infty)$ by
\[
c_x:=K-\frac\delta4
\qquad(x\in P_1).
\]
For every nonempty set $S\subseteq P_1$,
\[
|E_{G[P_1]}(S)|
=
|E(S)|
\le
(K-\delta)|S|
<
\left(K-\frac\delta4\right)|S|
=
\sum_{x\in S}c_x.
\]
Thus the hypothesis of \cref{lem:finite-logit} is satisfied, and there
exists a potential $u_1:P_1\to\R$ such that
\[
L_x^{G[P_1]}(u_1)
\le
K-\frac\delta4
\qquad(x\in P_1).
\]

Suppose that $u_n:P_n\to\R$ ($n\ge1$) has been constructed. For each
$x\in P_n$, set
\[
\varepsilon_{n,x}:=2^{-n-2}\delta.
\]
Apply \cref{lem:frozen-extension} with
\[
P=P_n,
\qquad
H=G[P_{n+1}],
\qquad
u^P=u_n.
\]
This gives a potential
\[
u_{n+1}:P_{n+1}\to\R
\]
such that
\[
u_{n+1}|_{P_n}=u_n,
\]
\[
L_x^{G[P_{n+1}]}(u_{n+1})
\le
K-\frac\delta4
\qquad(x\in P_{n+1}\setminus P_n),
\]
and
\[
L_x^{G[P_{n+1}]}(u_{n+1})
-
L_x^{G[P_n]}(u_n)
\le
2^{-n-2}\delta
\qquad(x\in P_n).
\]

Since the potentials agree on the nested sets, define
\[
u(x):=u_n(x)
\]
for any $n\ge1$ such that $x\in P_n$. This is well defined and gives a potential $u:V(G)\to\R$.

Let $x\in V(G)$, and let $m\ge1$ be the smallest index such that
$x\in P_m$. At stage $m$,
\[
L_x^{G[P_m]}(u_m)
\le
K-\frac\delta4.
\]
For every $n\ge m$, recall that
\[
L_x^{G[P_{n+1}]}(u_{n+1})
-
L_x^{G[P_n]}(u_n)
\le
2^{-n-2}\delta.
\]
Since $G$ is locally finite, $x$ has only finitely many neighbors.
As $\bigcup_{n\ge1}P_n=V(G)$, there exists $N$ such that all neighbors
of $x$ belong to $P_N$. Hence
\[
L_x^{G[P_n]}(u_n)=L_x^G(u)
\qquad(n\ge N).
\]
Summing the preceding inequalities from $m$ to $N$ gives
\[
L_x^G(u)
\le
K-\frac\delta4
+
\sum_{n=m}^{N-1}2^{-n-2}\delta\le K-\frac\delta4
+
\sum_{n=m}^{\infty}2^{-n-2}\delta
\le
K.
\]
Thus
\[
L_x^G(u)\le K
\qquad(x\in V(G)),
\]
as required.
\end{proof}

\subsection{Exact reduction to positive edge weights}

\begin{theorem}[Logit reduction]
\label{thm:logit-reduction}
Let $G$ be a connected locally finite graph with $E(G)\ne\varnothing$. For every $\lambda>0$, the following are equivalent:
\begin{enumerate}[label=\textup{(\roman*)}]
\item there is a positive edge weight $w$ such that
\[
q_e^w\ge\lambda
\qquad(e\in E(G));
\]
\item there is a potential $u:V(G)\to\R$ such that
\[
L_x^G(u)\le\frac1\lambda
\qquad(x\in V(G)).
\]
\end{enumerate}
\end{theorem}

\begin{proof}
Suppose first that $w>0$ satisfies \textup{(i)}. For any edge $e=xy\in E(G)$, put
\[
D_x:=d_x^w,
\qquad
p_{x,xy}:=\frac{w_{xy}}{D_x},
\]
and define $u:V(G)\to\R$ by
\[
u_x:=\log D_x\qquad (x\in V(G)).
\]
Then
\[
\sum_{y\sim x}p_{x,xy}=1,
\qquad
q_{xy}^w=p_{x,xy}+p_{y,xy}.
\]
Set
\[
r_{x,xy}:=\frac{p_{x,xy}}{q_{xy}^w}.
\]
A direct calculation gives
\[
r_{x,xy}
=
\frac{D_y}{D_x+D_y}
=
\sigma(u_y-u_x).
\]
Therefore
\[
L_x^G(u)
=
\sum_{y\sim x}\frac{p_{x,xy}}{q_{xy}^w}
\le
\frac1\lambda
\sum_{y\sim x}p_{x,xy}
=
\frac1\lambda.
\]

Conversely, suppose \textup{(ii)} holds. For any edge $e=xy\in E(G)$, put
\[
t_x:=e^{-u_x},
\qquad
r_{x,xy}:=\frac{t_x}{t_x+t_y},
\]
and define $w:E(G)\to (0,\infty)$ by
\[
w_{xy}:=\frac1{t_x+t_y}\qquad(e=xy\in E(G)).
\]
Then
\[
r_{x,xy}
=
\sigma(u_y-u_x),
\qquad
r_{x,xy}+r_{y,xy}=1,
\]
and
\[
d_x^w
=
\frac1{t_x}L_x^G(u).
\]
Since $G$ is connected and $E(G)\ne\varnothing$, one has
$L_x^G(u)>0$. Hence
\[
\frac{w_{xy}}{d_x^w}
=
\frac{r_{x,xy}}{L_x^G(u)}.
\]
It follows that
\[
q_{xy}^w
=
\frac{r_{x,xy}}{L_x^G(u)}
+
\frac{r_{y,xy}}{L_y^G(u)}
\ge
\lambda\bigl(r_{x,xy}+r_{y,xy}\bigr)
=
\lambda.
\]
\end{proof}

\begin{theorem}[Exact density optimization]
\label{thm:exact-density}
For every connected locally finite graph $G$ with $E(G)\ne\varnothing$,
\begin{equation}
\label{eq:exact-density}
\mathcal Q(G)
=
\frac2{\mad(G)}.
\end{equation}
\end{theorem}

\begin{proof}
The upper bound is \cref{prop:density-upper}. If
$\mad(G)=\infty$, then \eqref{eq:density-upper} and
$\mathcal Q(G)\ge0$ give $\mathcal Q(G)=0$.

Assume $\mad(G)<\infty$ and fix
\[
0<\lambda<\frac2{\mad(G)}.
\]
Then
\[
K:=\frac1\lambda>\frac{\mad(G)}2.
\]
By \cref{thm:uniform-logit}, there is a potential $u$
with
\[
L_x^G(u)\le K
\qquad(x\in V(G)).
\]
By \cref{thm:logit-reduction}, there is a positive edge weight $w$
such that
\[
q_e^w\ge\lambda
\qquad(e\in E(G)).
\]
Thus $\mathcal Q(G)\ge\lambda$. Letting
$\lambda\uparrow2/\mad(G)$ proves \eqref{eq:exact-density}.
\end{proof}

\begin{proof}[Proof of \cref{thm:intro-high-girth}]
By \eqref{eq:kmax-Q} and \cref{thm:exact-density},
\[
\Kmax(G)
=
2\mathcal Q(G)-2
=
\frac4{\mad(G)}-2.
\]

For every finite subgraph $H\subseteq G$ with $E(H)\ne\varnothing$,
\[
c(H)-\beta(H)
=
|V(H)|-|E(H)|,
\]
and hence
\[
\frac{2(c(H)-\beta(H))}{|E(H)|}
=
2\frac{|V(H)|}{|E(H)|}-2.
\]
By the definition of $\mad(G)$,
\[
\inf_{\substack{H\subseteq G\text{ finite}\\E(H)\ne\varnothing}}
\frac{|V(H)|}{|E(H)|}
=
\frac2{\mad(G)},
\]
with the convention $1/\infty=0$. Therefore
\[
\inf_{\substack{H\subseteq G\text{ finite}\\E(H)\ne\varnothing}}
\frac{2(c(H)-\beta(H))}{|E(H)|}
=
\frac4{\mad(G)}-2
=
\Kmax(G),
\]
which proves \eqref{eq:intro-euler}.

If $G$ is finite, there are only finitely many subgraphs
$H\subseteq G$ with $E(H)\ne\varnothing$, so the infimum in
\eqref{eq:intro-euler} is attained and is therefore a minimum.
\end{proof}

\subsection{Attainment and densest cores}

For the remainder of this subsection, $G$ is finite and
$\girth(G)\ge6$. Strict balancedness is equivalent to
\begin{equation}
\frac{|E(G[\Omega])|}{|\Omega|}
<
\frac{|E(G)|}{|V(G)|}
\qquad
\text{for every }\varnothing\ne\Omega\subsetneq V(G).
\end{equation}

\begin{proof}[Proof of \cref{thm:intro-attainment}]
Suppose first that $G$ is strictly balanced. Then $G$ itself uniquely
maximizes the edge--vertex density, and hence
\[
\mad(G)=\frac{2|E(G)|}{|V(G)|}.
\]
By \cite[Theorem~3.1]{LinLiu2026}, there exists a positive edge
weight $w$ such that
\[
\kappa_{\LLY}^w(e)=\bar\kappa
:=
2\left(\frac{|V(G)|}{|E(G)|}-1\right)
\qquad(e\in E(G)).
\]
Using \cref{thm:intro-high-girth},
\[
\bar\kappa
=
\frac{2|V(G)|}{|E(G)|}-2
=
\frac4{\mad(G)}-2
=
\Kmax(G),
\]
so $w$ is maximizing.

If $w'$ is any maximizing weight, then
\[
\kappa_{\LLY}^{w'}(e)\ge\Kmax(G)=\bar\kappa
\qquad(e\in E(G)).
\]
The total-curvature identity \cite[Proposition~3.1]{LinLiu2026}
gives
\[
\sum_{e\in E(G)}\kappa_{\LLY}^{w'}(e)
=
2(|V(G)|-|E(G)|)
=
|E(G)|\bar\kappa.
\]
Hence equality holds on every edge. The injectivity result
\cite[Proposition~3.2]{LinLiu2026} shows that $w'$ differs from $w$
only by a global scaling factor.

Conversely, suppose that $w$ is maximizing and put
\[
t_*:=\frac2{\mad(G)}=\mathcal Q(G).
\]
By \eqref{eq:local-high-girth},
\[
\min_{e\in E(G)}q_e^w=t_*,
\]
so $q_e^w\ge t_*$ for every edge. Let $F\subseteq E(G)$ be a
nonempty edge set of maximum edge--vertex density. Then
\[
\frac{|V(F)|}{|F|}=t_*.
\]
By \cref{lem:edge-set-q},
\[
t_*|F|
\le
\sum_{e\in F}q_e^w
=
|V(F)|-Leak_w(G_F)
\le
|V(F)|
=
t_*|F|.
\]
Thus $Leak_w(G_F)=0$, and \cref{lem:edge-set-q} gives $F=E(G)$.
Therefore $E(G)$ is the unique nonempty edge set of maximum density,
which is equivalent to strict balancedness.
\end{proof}

A nonempty subgraph $H\subseteq G$ is called a \emph{densest core}
if
\[
\frac{2|E(H)|}{|V(H)|}=\mad(G).
\]
A densest core has no isolated vertices and is induced on its vertex
set. If $G$ is not strictly balanced, then it has a proper densest
core.

\begin{proposition}[Densest-core concentration]
\label{prop:core-concentration}
Let $H\subsetneq G$ be a densest core and put
\[
t_*:=\frac{|V(H)|}{|E(H)|}=\frac2{\mad(G)}.
\]
For a positive edge weight $w$, set
\[
\lambda(w):=\min_{e\in E(G)}q_e^w.
\]
Then
\begin{equation}
\label{eq:leakage-bound}
0<Leak_w(H)
\le
|E(H)|\bigl(t_*-\lambda(w)\bigr).
\end{equation}
Consequently, if $\{w^{(j)}\}$ is a maximizing sequence, then
\[
Leak_{w^{(j)}}(H)\longrightarrow0
\]
and
\[
q_e^{w^{(j)}}\longrightarrow t_*
\qquad(e\in E(H)).
\]
\end{proposition}

\begin{proof}
By \cref{lem:edge-set-q},
\[
\sum_{e\in E(H)}q_e^w
=
|V(H)|-Leak_w(H).
\]
Since $q_e^w\ge\lambda(w)$,
\[
|E(H)|\lambda(w)
\le
|V(H)|-Leak_w(H)
=
|E(H)|t_*-Leak_w(H),
\]
which proves the upper bound in \eqref{eq:leakage-bound}. Since $H$
is proper and the weights are positive, \cref{lem:edge-set-q} gives
$Leak_w(H)>0$.

By \eqref{eq:local-high-girth},
\[
\kappa_{\LLY}^{w^{(j)}}(G)
=
2\lambda(w^{(j)})-2.
\]
For a maximizing sequence, \cref{thm:intro-high-girth} therefore
gives $\lambda(w^{(j)})\to t_*$ and hence
$Leak_{w^{(j)}}(H)\to0$. Finally,
\[
\sum_{e\in E(H)}
\bigl(q_e^{w^{(j)}}-\lambda(w^{(j)})\bigr)
=
|V(H)|-Leak_{w^{(j)}}(H)-|E(H)|\lambda(w^{(j)})
\longrightarrow0.
\]
Every summand is nonnegative, so each tends to zero.
\end{proof}

\subsection{Infinite graphs and endpoint regimes}

We next record the consequences that are specific to infinite graphs.

\begin{proposition}[Universal range for infinite graphs]
\label{prop:infinite-universal-range}
Let $G$ be connected, locally finite, and infinite. Then
\begin{equation}
-2\le\Kmax(G)\le0.
\end{equation}
\end{proposition}

\begin{proof}
By \cref{lem:baseline-curvature},
\[
\kappa_{\LLY}^w(e)
\ge
2q_e^w-2
>-2
\qquad(e\in E(G)).
\]
Hence $\kappa_{\LLY}^w(G)\ge-2$ for every positive weight and
$\Kmax(G)\ge-2$.

If $\kappa_{\LLY}^w(G)\ge k>0$, the
weighted discrete Bonnet--Myers theorem
\cite[Proposition~1.2]{Xia2026CDNP}
for Lin--Lu--Yau curvature gives
\[
\operatorname{diam}(G)\le\frac2k;
\]
see also \cite{MunchWojciechowski2019}. Hence finite diameter would imply
that $G$ is finite. Therefore $\kappa_{\LLY}^w(G)\le0$ for every
positive weight, and $\Kmax(G)\le0$.
\end{proof}

In the high-girth setting, \cref{thm:intro-high-girth} further gives
\[
\Kmax(G)=0\iff\mad(G)=2,\qquad\Kmax(G)=-2\iff\mad(G)=\infty.
\]

\begin{remark}
\label{rem:mad-two-infinite}
For an infinite connected locally finite graph $G$,
\[
\mad(G)=2
\]
if and only if $G$ is a tree or a unicyclic graph, i.e., a connected graph containing
exactly one cycle. Indeed, every
finite subgraph of a tree or a unicyclic graph satisfies
$|E(H)|\le|V(H)|$, while arbitrarily long finite paths show that
$\mad(G)\ge2$. Conversely, if $G$ contains at least two cycles, then
some finite connected subgraph $H\subseteq G$ has
$\beta(H)\ge2$, and hence
\[
|E(H)|
=
|V(H)|-1+\beta(H)
>
|V(H)|,
\]
so $\mad(G)>2$.
\end{remark}

\section{Short-cycle surplus and girth-density bounds}
\label{sec:short-cycle}

The baseline term $2q_e^w-2$ is the curvature on edges of local girth
at least $6$. For an arbitrary edge, define the \emph{short-cycle
surplus} by
\begin{equation}
\label{eq:surplus}
\sigma_e(w)
:=
\kappa_{\LLY}^w(e)-\bigl(2q_e^w-2\bigr).
\end{equation}
By \cref{lem:baseline-curvature},
\begin{equation}
\sigma_e(w)\ge0,
\qquad
\sigma_e(w)=0
\Longleftrightarrow
\girth_G(e)\ge6.
\end{equation}
Thus the surplus records the contribution of cycles of length at most
$5$ to the local curvature.

If $G$ is finite, summing \eqref{eq:surplus} and using
\cref{lem:edge-set-q} with $F=E(G)$ gives
\begin{equation}
\sum_{e\in E(G)}\kappa_{\LLY}^w(e)
=
2(|V(G)|-|E(G)|)
+
\sum_{e\in E(G)}\sigma_e(w).
\end{equation}
For $\girth(G)\ge6$, this reduces to the total-curvature identity of
\cite[Proposition~3.1]{LinLiu2026}.

\subsection{Triangle defect and the density hierarchy}

For an edge $e=xy$, define its normalized \emph{triangle defect} by
\begin{equation}
\eta_e(w)
:=
\sum_{\substack{u\sim x,\ u\ne y\\u\not\sim y}}
\frac{w_{xu}}{d_x^w}
+
\sum_{\substack{v\sim y,\ v\ne x\\v\not\sim x}}
\frac{w_{yv}}{d_y^w}.
\end{equation}
Thus $\eta_e(w)$ is the normalized incidence mass at the endpoints of
$e$ carried by neighbors that do not complete a triangle with $e$.

\begin{theorem}[Triangle-defect estimate]
\label{thm:triangle-defect}
For every positive edge weight and every edge $e\in E(G)$,
\begin{equation}
\label{eq:triangle-defect-bound}
\kappa_{\LLY}^w(e)
\le
1+\frac12q_e^w-\frac12\eta_e(w).
\end{equation}
In particular,
\[
\kappa_{\LLY}^w(e)\le1+\frac12q_e^w.
\]
\end{theorem}

\begin{proof}
Write $e=xy$ and put $C_e:=N(x)\cap N(y)$. Normalize
$f(x)=0$, $f(y)=1$ in \eqref{eq:dual-LLY}. Prescribe
\[
f(z)=\frac12
\quad(z\in C_e),
\]
\[
f(u)=0
\quad
\bigl(u\in N(x)\setminus(C_e\cup\{y\})\bigr),
\qquad
f(v)=1
\quad
\bigl(v\in N(y)\setminus(C_e\cup\{x\})\bigr).
\]
All prescribed values lie in $[0,1]$, so the prescription is
$1$-Lipschitz and extends to $V(G)$.

Write
\[
p:=\frac{w_{xy}}{d_x^w},
\qquad
r:=\frac{w_{xy}}{d_y^w},
\]
and
\[
A:=\sum_{z\in C_e}\frac{w_{xz}}{d_x^w},
\qquad
B:=\sum_{z\in C_e}\frac{w_{yz}}{d_y^w}.
\]
The dual expression at this test function equals
\[
p+r+\frac{A+B}{2}.
\]
Since
\[
\eta_e(w)=2-(p+r)-(A+B),
\]
this is exactly
$1+(p+r)/2-\eta_e(w)/2$, proving
\eqref{eq:triangle-defect-bound}.
\end{proof}

The local estimate also yields concentration on a finite densest edge
set whenever such a set exists.

\begin{proposition}[Triangle-defect concentration]
Let $F\subseteq E(G)$ be finite and nonempty. Assume
\[
\frac{2|F|}{|V(F)|}=\mad(G)<\infty.
\]
Then
\begin{equation}
\label{eq:triangle-defect-concentration}
Leak_w(G_F)+\sum_{e\in F}\eta_e(w)
\le
2|F|\left(
1+\frac1{\mad(G)}-\kappa_{\LLY}^w(G)
\right).
\end{equation}
Consequently, if
\[
\Kmax(G)=1+\frac1{\mad(G)}
\]
and $\{w^{(j)}\}$ is a maximizing sequence, then
\[
Leak_{w^{(j)}}(G_F)\to0,
\qquad
\eta_e(w^{(j)})\to0
\quad(e\in F).
\]
\end{proposition}

\begin{proof}
For every $e\in F$, \cref{thm:triangle-defect} gives
\[
2\bigl(\kappa_{\LLY}^w(G)-1\bigr)
\le
q_e^w-\eta_e(w).
\]
Summing over $F$ and applying \cref{lem:edge-set-q},
\[
2|F|\bigl(\kappa_{\LLY}^w(G)-1\bigr)
\le
|V(F)|-Leak_w(G_F)-\sum_{e\in F}\eta_e(w).
\]
Since $2|F|/|V(F)|=\mad(G)$, rearranging gives
\eqref{eq:triangle-defect-concentration}. The final assertion follows
from the definition of a maximizing sequence.
\end{proof}

If $e$ lies on no triangle, then
$\eta_e(w)=2-q_e^w$, so \eqref{eq:triangle-defect-bound} reduces to
\[
\kappa_{\LLY}^w(e)\le q_e^w.
\]

\begin{lemma}[Local girth bound]
\label{lem:local-girth}
Let $e\in E(G)$ and let $\ell\in\{4,5,6\}$. If
$\girth_G(e)\ge\ell$, then
\begin{equation}
\kappa_{\LLY}^w(e)
\le
4-\ell+\frac{\ell-2}{2}q_e^w.
\end{equation}
\end{lemma}

\begin{proof}
Write $e=xy$ and normalize $f(x)=0$, $f(y)=1$. Prescribe
\[
f(u)=\frac{4-\ell}{2}
\quad(u\in N(x)\setminus\{y\}),
\]
\[
f(v)=\frac{\ell-2}{2}
\quad(v\in N(y)\setminus\{x\}).
\]
Since $\ell\ge4$, the two neighbor sets are disjoint. The condition
$\girth_G(e)\ge\ell$ gives $d(u,v)\ge\ell-3$ for points on opposite
sides, while the difference of the two prescribed values is exactly
$\ell-3$. Thus the prescription is $1$-Lipschitz and extends to
$V(G)$. Evaluating the dual expression gives
\[
\Delta_wf(x)-\Delta_wf(y)
=
4-\ell+\frac{\ell-2}{2}q_e^w.
\]
Taking the infimum proves the claim.
\end{proof}

\begin{proof}[Proof of \cref{thm:intro-girth-density}]
Fix a positive edge weight $w$ and a finite nonempty edge set
$F\subseteq E(G)$. By \cref{lem:edge-set-q}, some $e_0\in F$ satisfies
\[
q_{e_0}^w
\le
\frac{|V(F)|}{|F|}.
\]
If $\ell=3$, \cref{thm:triangle-defect} gives
\[
\kappa_{\LLY}^w(G)
\le
\kappa_{\LLY}^w(e_0)
\le
1+\frac12\frac{|V(F)|}{|F|}.
\]
For $\ell\in\{4,5,6\}$, \cref{lem:local-girth} gives
\[
\kappa_{\LLY}^w(G)
\le
4-\ell
+
\frac{\ell-2}{2}\frac{|V(F)|}{|F|}.
\]
Taking the infimum over nonempty finite $F$ and using \eqref{eq:mad-edge-set} proves \eqref{eq:intro-girth-density} after taking the supremum over
$w$. If $G$ is infinite, the additional
bound \eqref{eq:intro-infinite-girth-density} follows from
\cref{prop:infinite-universal-range}.
\end{proof}

\subsection{Matching surplus at girth four and five}

Let $e=xy$ satisfy $\girth_G(e)\ge4$, and set
\[
X_e:=N(x)\setminus\{y\},
\qquad
Y_e:=N(y)\setminus\{x\},
\]
and
\begin{equation}
a_u:=\frac{w_{xu}}{d_x^w}
\quad(u\in X_e),
\qquad
b_v:=\frac{w_{yv}}{d_y^w}
\quad(v\in Y_e).
\end{equation}

\begin{lemma}[Local surplus formula]
\label{lem:matching-surplus}
Let $e=xy\in E(G)$ satisfy $\girth_G(e)\ge4$. Then
\begin{equation}
\label{eq:surplus-primal}
\sigma_e(w)
=
\min\left\{
\sum_{u\in X_e}a_ug_u+\sum_{v\in Y_e}b_vh_v:
\begin{array}{l}
g_u,h_v\ge0,\\
g_u+h_v\ge3-d(u,v)
\end{array}
\right\}.
\end{equation}
Equivalently,
\begin{equation}
\label{eq:surplus-dual}
\sigma_e(w)
=
\max_{\gamma\ge0}
\sum_{u\in X_e,v\in Y_e}
\bigl(3-d(u,v)\bigr)\gamma_{uv},
\end{equation}
subject to
\begin{equation}
\label{eq:gamma-capacity}
\sum_{v\in Y_e}\gamma_{uv}\le a_u,
\qquad
\sum_{u\in X_e}\gamma_{uv}\le b_v.
\end{equation}
\end{lemma}

\begin{proof}
Let $A:=\sigma_e(w)$ and let $B$ denote the right-hand side of
\eqref{eq:surplus-primal}. First, let
$f\in\operatorname{Lip}(1)$ satisfy $f(x)=0$, $f(y)=1$, and set
\[
g_u:=f(u)+1,
\qquad
h_v:=2-f(v).
\]
Then $g_u,h_v\ge0$, and
\[
\Delta_wf(x)-\Delta_wf(y)
=
2q_e^w-2
+
\sum_{u\in X_e}a_ug_u
+
\sum_{v\in Y_e}b_vh_v.
\]
For $u\in X_e$ and $v\in Y_e$, the Lipschitz condition gives
\[
3-g_u-h_v=f(v)-f(u)\le d(u,v),
\]
so $(g,h)$ is feasible for $B$. Taking the infimum over $f$ gives
$A\ge B$.

Conversely, let $(g,h)$ be a minimizer for $B$. Since $G$ is locally
finite, the optimization problem is finite-dimensional. Define
\[
\widehat g_u
:=
\max\Bigl(
\{0\}\cup
\{3-d(u,v)-h_v:v\in Y_e\}
\Bigr).
\]
Then $\widehat g\le g$, and $(\widehat g,h)$ is feasible. For each
fixed $v$, the function
\[
u\longmapsto3-d(u,v)-h_v
\]
is $1$-Lipschitz. Hence $u\mapsto\widehat g_u$ is $1$-Lipschitz.
Next define
\[
\widehat h_v
:=
\max\Bigl(
\{0\}\cup
\{3-d(u,v)-\widehat g_u:u\in X_e\}
\Bigr).
\]
Then $\widehat h\le h$, $(\widehat g,\widehat h)$ is feasible, and
$v\mapsto\widehat h_v$ is $1$-Lipschitz. Moreover,
\[
0\le\widehat g_u,\widehat h_v\le2.
\]
Since all coefficients in the objective are nonnegative,
$(\widehat g,\widehat h)$ is also a minimizer.

Define
\[
f(x)=0,
\qquad
f(y)=1,
\]
and
\[
f(u)=-1+\widehat g_u,
\qquad
f(v)=2-\widehat h_v.
\]
The preceding properties give the Lipschitz conditions within $X_e$
and within $Y_e$. For $u\in X_e$ and $v\in Y_e$,
\[
f(v)-f(u)
=
3-\widehat g_u-\widehat h_v
\le
d(u,v),
\]
while
\[
f(u)-f(v)
=
\widehat g_u+\widehat h_v-3
\le
1
\le
d(u,v).
\]
The bounds on $\widehat g$ and $\widehat h$ give the remaining
Lipschitz conditions involving $x$ and $y$. The local assignment
therefore extends to $V(G)$. Evaluating the dual expression gives
$A\le B$, and hence $A=B$.

The dual formula \eqref{eq:surplus-dual}--\eqref{eq:gamma-capacity}
follows from linear-programming duality
\cite[Theorem~5.4, Eq.~(5.11)]{Schrijver2003}.
\end{proof}

In \eqref{eq:surplus-dual}, call $3-d(u,v)$ the \emph{profit} of
$(u,v)\in X_e\times Y_e$. Profit $2$ corresponds to a $4$-cycle
through $e$, profit $1$ to a $5$-cycle through $e$, and profit $0$ to
distance $3$.

Recall that, for $\ell\in\{4,5\}$ and an edge $e=xy$ with
$\girth_G(e)\ge\ell$, the auxiliary bipartite graph
$\mathcal M_\ell(e)$ has parts $X_e$ and $Y_e$, with
\[
u\sim_{\mathcal M_\ell(e)}v
\quad\Longleftrightarrow\quad
d(u,v)=\ell-3.
\]
Thus an auxiliary edge corresponds to a cycle of length $\ell$
through $e$.

\begin{lemma}[Capacitated Hall criterion]
\label{lem:capacitated-hall}
Let $\mathcal M$ be a finite bipartite graph with parts $X$ and $Y$,
and let $(\alpha_u)_{u\in X}$ and $(\beta_v)_{v\in Y}$ be
nonnegative numbers satisfying
\[
\sum_{u\in X}\alpha_u
=
\sum_{v\in Y}\beta_v.
\]
For $S\subseteq X$, write
\[
N_{\mathcal M}(S)
:=
\{v\in Y:\text{ there exists }u\in S
\text{ with }uv\in E(\mathcal M)\}.
\]
Then the following are equivalent:
\begin{enumerate}[label=\textup{(\roman*)}]
\item
\[
\sum_{u\in S}\alpha_u
\le
\sum_{v\in N_{\mathcal M}(S)}\beta_v
\qquad
\text{for every }S\subseteq X.
\]
\item There exist numbers $\gamma_{uv}\ge0$ $(u\in X,v\in Y)$,
with $\gamma_{uv}=0$ whenever $uv\notin E(\mathcal M)$, such that
\[
\sum_{v\in Y}\gamma_{uv}=\alpha_u
\quad(u\in X),
\qquad
\sum_{u\in X}\gamma_{uv}=\beta_v
\quad(v\in Y).
\]
\end{enumerate}
\end{lemma}

\begin{proof}
This is the standard capacitated form of Hall's criterion and follows
from the max-flow--min-cut theorem
\cite[Theorem~10.3]{Schrijver2003}.
\end{proof}

\begin{lemma}[Local equality at girth four and five]
\label{lem:girth45-local-equality}
Let $w$ be a positive edge weight and let $e=xy\in E(G)$.
\begin{enumerate}[label=\textup{(\roman*)}]
\item If $\girth_G(e)\ge4$, then
\[
\kappa_{\LLY}^w(e)\le q_e^w.
\]
Equality holds if and only if $d_x^w=d_y^w$ and
\[
\sum_{u\in S}a_u
\le
\sum_{v\in N_{\mathcal M_4(e)}(S)}b_v
\qquad
\text{for every }S\subseteq X_e.
\]
\item If $\girth_G(e)\ge5$, then
\[
\kappa_{\LLY}^w(e)\le\frac32q_e^w-1.
\]
Equality holds if and only if $d_x^w=d_y^w$ and
\[
\sum_{u\in S}a_u
\le
\sum_{v\in N_{\mathcal M_5(e)}(S)}b_v
\qquad
\text{for every }S\subseteq X_e.
\]
\end{enumerate}
\end{lemma}

\begin{proof}
Write
\[
p:=\frac{w_{xy}}{d_x^w},
\qquad
r:=\frac{w_{xy}}{d_y^w}.
\]
For a feasible $\gamma$ in \eqref{eq:surplus-dual}, put
\[
T(\gamma)
:=
\sum_{u\in X_e}\sum_{v\in Y_e}\gamma_{uv}.
\]
By \eqref{eq:gamma-capacity},
\begin{equation}
\label{eq:gamma-total-bound}
T(\gamma)\le\min\{1-p,1-r\}.
\end{equation}
The feasible set is compact, so the maximum in
\eqref{eq:surplus-dual} is attained.

Suppose first that $\girth_G(e)\ge4$. Every profit is at most $2$, so
\cref{lem:matching-surplus} and \eqref{eq:gamma-total-bound} give
\[
\sigma_e(w)
\le
2\min\{1-p,1-r\}
\le
2-p-r.
\]
Therefore
\[
\kappa_{\LLY}^w(e)
=
2(p+r)-2+\sigma_e(w)
\le
p+r
=
q_e^w.
\]
If equality holds, an optimizer $\gamma$ satisfies
\[
p=r,
\qquad
T(\gamma)=1-p=1-r,
\]
and every pair with $\gamma_{uv}>0$ has profit $2$. Hence $\gamma$
is supported on $\mathcal M_4(e)$ and saturates all capacities.
Since
\[
\sum_{u\in X_e}a_u
=
\sum_{v\in Y_e}b_v
=
1-p,
\]
\cref{lem:capacitated-hall} gives the stated Hall inequalities. Also,
$p=r$ is equivalent to $d_x^w=d_y^w$.

Conversely, suppose $d_x^w=d_y^w$ and the Hall inequalities hold.
Then $p=r$, and \cref{lem:capacitated-hall} gives a feasible
$\gamma$ supported on $\mathcal M_4(e)$ that saturates all
capacities. Every auxiliary edge has profit $2$, so
\[
\sigma_e(w)\ge2(1-p).
\]
Together with the upper bound this gives
$\kappa_{\LLY}^w(e)=q_e^w$, proving \textup{(i)}.

Now suppose that $\girth_G(e)\ge5$. Every positive profit is $1$, so
\[
\sigma_e(w)
\le
\min\{1-p,1-r\}
\le
1-\frac{p+r}{2},
\]
and therefore
\[
\kappa_{\LLY}^w(e)
\le
\frac32(p+r)-1
=
\frac32q_e^w-1.
\]
If equality holds, an optimizer satisfies
\[
p=r,
\qquad
T(\gamma)=1-p=1-r,
\]
and is supported on $\mathcal M_5(e)$ with all capacities saturated.
The Hall inequalities follow from \cref{lem:capacitated-hall}.
Conversely, if $d_x^w=d_y^w$ and the Hall inequalities hold, the same
criterion gives a feasible $\gamma$ supported on $\mathcal M_5(e)$
that saturates all capacities. Since every auxiliary edge has profit
$1$,
\[
\sigma_e(w)\ge1-p,
\]
and equality follows. This proves \textup{(ii)}.
\end{proof}

\section{Finite rigidity at girth three, four, and five}
\label{sec:equality}

\begin{proof}[Proof of \cref{thm:girth3-rigidity}]
Suppose first that a positive weight $w$ attains the upper bound. Choose a nonempty edge set $F\subseteq E(G)$ with
\[
\frac{|F|}{|V(F)|}=\frac{\mad(G)}2.
\]
Then \eqref{eq:triangle-defect-concentration} has zero right-hand side, so
\[
Leak_w(G_F)=0,
\qquad
\eta_e(w)=0
\quad(e\in F).
\]
By the equality characterization in \cref{lem:edge-set-q}, positivity and connectedness imply $F=E(G)$. Hence $\eta_e(w)=0$ for every edge $e=xy$. Thus, for every edge $xy$, every neighbor of $x$ other than $y$ is adjacent to $y$, and conversely. If $G$ were not complete, there would exist a shortest path $v_0v_1\cdots v_k$ with $k\ge2$. Since $v_2\sim v_1$ and $v_0\sim v_1$, the preceding property implies $v_0\sim v_2$, a contradiction. Hence $G$ is complete.

Conversely, let $G=K_n$ ($n\ge2$). For the unit weight, we have
\[
\kappa_{\LLY}(e)=\frac{n}{n-1}
\qquad(e\in E(G)).
\]
Since $\mad(K_n)=n-1$, the universal bound is attained.

It remains to prove uniqueness. Let $w$ be a maximizing weight on $K_n$ ($n\ge2$). Since $\eta_e(w)=0$ for every edge, \cref{thm:triangle-defect} gives \[ \kappa_{\LLY}^w(e)\le1+\frac12q_e^w. \] Together with \[ \sum_{e\in E(G)}q_e^w=|V(G)| \] and $\kappa_{\LLY}^w(G)=n/(n-1)$, this forces \[ q_e^w=\frac{2}{n-1} \] and equality in the triangle-defect estimate for every edge. For $e=xy$, equality in the triangle-defect estimate means that the choice \[ f(z)=\frac12\qquad(z\ne x,y) \] minimizes the dual expression. On $K_n$, each $f(z)$ may vary independently in $[0,1]$, and the dual
expression is affine in $f(z)$ with coefficient
\[
\frac{w_{xz}}{d_x^w}-\frac{w_{yz}}{d_y^w}.
\]
Since equality in the triangle-defect estimate is attained at
$f(z)=1/2$, an interior point of $[0,1]$, this coefficient must vanish. Hence \[ \frac{w_{xz}}{d_x^w} = \frac{w_{yz}}{d_y^w} \qquad \text{for all distinct }x,y,z. \] For $n\ge3$, denote this common value, for fixed $z$, by $c_z$. Then \[ 1=\sum_{z\ne x}\frac{w_{xz}}{d_x^w}=\sum_{z\ne x}c_z\qquad(x\in V(G)), \] so all $c_z$ are equal. For fixed $x$ and distinct $y,z\ne x$, \[ \frac{w_{xz}}{d_x^w} = c_z = c_y = \frac{w_{xy}}{d_x^w}, \] so $w_{xz}=w_{xy}$. Thus all edges incident with the same vertex have the same weight. Since $K_n$ is connected, all edge weights are equal. The case $n=2$ is immediate. Thus the maximizing weight is unique up to global scaling.
\end{proof}

The proof of the girth-$4$ and girth-$5$ rigidity statement uses
\cref{lem:girth45-local-equality}.

\begin{proof}[Proof of \cref{thm:girth45-rigidity}]
Put
\[
t_*:=\frac2{\mad(G)},
\qquad
k_*:=4-\ell+\frac{\ell-2}{2}t_*.
\]
Suppose that a positive weight $w$ satisfies
\[
\kappa_{\LLY}^w(G)=k_*.
\]
Choose a nonempty $F\subseteq E(G)$ such that
\[
\frac{2|F|}{|V(F)|}=\mad(G).
\]
For every $e\in F$, \cref{lem:girth45-local-equality} gives
\[
k_*
\le
\kappa_{\LLY}^w(e)
\le
4-\ell+\frac{\ell-2}{2}q_e^w,
\]
so $q_e^w\ge t_*$. Hence
\[
t_*|F|
\le
\sum_{e\in F}q_e^w
=
|V(F)|-Leak_w(G_F)
\le
|V(F)|
=
t_*|F|.
\]
Thus $Leak_w(G_F)=0$. By \cref{lem:edge-set-q}, $F=E(G)$, and equality in
the first inequality gives
\[
q_e^w=t_*
\qquad(e\in E(G)).
\]

Equality therefore holds in \cref{lem:girth45-local-equality} for every
$e=xy$. In particular,
\[
d_x^w=d_y^w
\qquad(x\sim y).
\]
Since $G$ is connected, there is a constant $D>0$ such that
\[
d_x^w=D
\qquad(x\in V(G)).
\]
Using $q_e^w=t_*$,
\[
\frac{2w_{xy}}{D}=t_*,
\]
so all edge weights are equal. Moreover,
\[
1
=
\sum_{y\sim x}\frac{w_{xy}}{D}
=
\frac{d_xt_*}{2},
\]
and hence
\[
d_x=\frac2{t_*}=\mad(G)
\qquad(x\in V(G)).
\]
Thus $G$ is $d$-regular, where $d=\mad(G)$.

For every $e=xy$,
\[
a_u=b_v=\frac1d
\qquad(u\in X_e,\ v\in Y_e).
\]
The Hall condition in \cref{lem:girth45-local-equality} therefore
becomes
\[
|S|
\le
|N_{\mathcal M_\ell(e)}(S)|
\qquad(S\subseteq X_e).
\]
Since
\[
|X_e|=|Y_e|=d-1,
\]
Hall's marriage theorem \cite{Hall1935} implies that
$\mathcal M_\ell(e)$ has a perfect matching.

Conversely, suppose that $G$ is $d$-regular and
$\mathcal M_\ell(e)$ has a perfect matching for every $e\in E(G)$.
Take the unit weight. Then
\[
q_e=\frac2d,
\qquad
a_u=b_v=\frac1d.
\]
A perfect matching gives
\[
|S|
\le
|N_{\mathcal M_\ell(e)}(S)|
\qquad(S\subseteq X_e),
\]
and hence
\[
\sum_{u\in S}a_u
\le
\sum_{v\in N_{\mathcal M_\ell(e)}(S)}b_v.
\]
Therefore \cref{lem:girth45-local-equality} gives
\[
\kappa_{\LLY}(e)
=
4-\ell+\frac{\ell-2}{d}
\qquad(e\in E(G)).
\]
Since $G$ is $d$-regular, $\mad(G)=d$, so this is the upper bound.
The first part of the proof also shows that every maximizing weight is
constant up to global scaling.
\end{proof}

\begin{example}[Complete bipartite graphs]
For $a,b\ge1$,
\begin{equation}
\Kmax(K_{a,b})
=
\frac2{\max\{a,b\}}.
\end{equation}
Indeed, let $w$ be any positive weight. For every edge $e=xy$ of $K_{a,b}$, all pairs in
$X_e\times Y_e$ have distance $1$. Hence \cref{lem:matching-surplus}
gives
\[
\sigma_e(w)=2\min\left\{1-\frac{w_{xy}}{d_x^w},\,1-\frac{w_{xy}}{d_y^w}\right\}.
\]
Hence
\[
\kappa_{\LLY}^w(x,y)
=
2\min\left\{
\frac{w_{xy}}{d_x^w},
\frac{w_{xy}}{d_y^w}
\right\}.
\]
Assume $a\le b$ and let $x$ belong to the part of size $a$, so
$d_x=b$. Some edge $xz$ incident with $x$ satisfies
\[
\frac{w_{xz}}{d_x^w}\le\frac1b,
\]
and therefore
\[
\kappa_{\LLY}^w(K_{a,b})\le\frac2b.
\]
For the unit weight,
\[
\kappa_{\LLY}(e)=\frac2b
\qquad(e\in E(K_{a,b})),
\]
so
\[
\Kmax(K_{a,b})=\frac2b.
\]
\end{example}

\begin{remark}
For $a,b\ge2$, the preceding example also shows that the girth-$4$
bound
\[\Kmax(K_{a,b})
\le
\frac{2}{\mad(K_{a,b})}
=
\frac1a+\frac1b\]
is attained if and only if $a=b$. Thus the family $K_{a,b}$ illustrates
the rigidity in \cref{thm:girth45-rigidity}: equality occurs precisely
in the regular case, while for $a\ne b$ the girth-$4$ bound is strict.
Moreover, when $a=1$, the formula reduces to
\[\Kmax(K_{1,b})=\frac2b,\]
in agreement with the tree formula \eqref{eq:intro-tree}.
\end{remark}

\section{Nonnegative curvature in high girth}
\label{sec:infinite}

The following tree construction is valid for finite and infinite
locally finite trees.

\begin{lemma}[Pointwise positive weights on trees]
\label{lem:tree-pointwise-positive}
Let $T$ be a connected locally finite tree with $E(T)\ne\varnothing$.
Then there exists a positive edge weight $w$ such that
\[
q_e^w>1
\qquad(e\in E(T)).
\]
Consequently,
\[
\kappa_{\LLY}^w(e)>0
\qquad(e\in E(T)).
\]
\end{lemma}

\begin{proof}
Choose a
probability measure
\[
\mu:V(T)\to(0,1),
\qquad
\mu(T):=\sum_{x\in V(T)}\mu(x)=1.
\]
For an edge $xy\in E(T)$, let $T_{x|y}$ and $T_{y|x}$ be the two components
of $T-xy$ containing $x$ and $y$, respectively, and put
\begin{equation}\label{eq:p_xye}
p_{x,e}
:=
\frac{\mu(T_{y|x})}{1-\mu(x)},
\qquad
p_{y,e}
:=
\frac{\mu(T_{x|y})}{1-\mu(y)}.    
\end{equation}
We also use the notation $p_{x,xy}$ when the endpoints are more
convenient than the edge label. For every vertex $x$,
\[
\sum_{f\ni x}p_{x,f}=1.
\]
Moreover, if $e=xy$, then
\[
p_{x,e}+p_{y,e}>1,
\]
because
\[
\mu(T_{x|y})+\mu(T_{y|x})=1
\]
and both denominators in \eqref{eq:p_xye} are strictly smaller than $1$.

Fix a root \(r\in V(T)\), set \(D_r=1\), and define positive numbers \(D_x\) recursively along the unique tree paths by
\[
\frac{D_y}{D_x}
=
\frac{p_{x,xy}}{p_{y,xy}},
\]
whenever \(y\) is a child of \(x\).
Set
\[
w_{xy}:=D_xp_{x,xy}=D_yp_{y,xy}.
\]
Then
\[
d_x^w
=
D_x\sum_{e\ni x}p_{x,e}
=
D_x,
\]
and therefore, for $e=xy$,
\[
q_e^w
=
p_{x,e}+p_{y,e}
>1.
\]
Since every edge of a tree has infinite local girth,
\eqref{eq:local-high-girth} gives
\[
\kappa_{\LLY}^w(e)=2q_e^w-2>0.
\]
\end{proof}

\begin{proof}[Proof of \cref{thm:intro-nonnegative-classification}]
Assume first that $w>0$ satisfies
\[
\kappa_{\LLY}^w(e)\ge0
\qquad(e\in E(G)).
\]
If $G$ contains a finite cycle $C$, then \eqref{eq:local-high-girth} gives
\[
q_e^w\ge1
\qquad(e\in E(C)).
\]
Since $|E(C)|=|V(C)|$, \cref{lem:edge-set-q} yields
\[
|E(C)|
\le
\sum_{e\in E(C)}q_e^w
=
|V(C)|-Leak_w(C)
\le
|V(C)|
=
|E(C)|.
\]
Hence $Leak_w(C)=0$. Since all edge weights are positive, no edge outside
$C$ is incident with a vertex of $C$. Connectedness therefore implies
$G=C$. Thus every such graph is either a tree or a finite cycle, and the girth assumption gives $C=C_n$ with $n\ge6$.

Conversely, by \cref{lem:tree-pointwise-positive}, every locally finite
tree admits a weight with positive curvature on every edge. If
$G=C_n$ with $n\ge6$, the constant weight has zero curvature on every
edge.

If $G$ is an infinite tree and $w$ induces nonnegative Lin--Lu--Yau curvature on every edge, \cref{prop:infinite-universal-range} excludes a positive
uniform lower bound. Hence
\[
\inf_{e\in E(G)}\kappa_{\LLY}^w(e)=0.
\]

Finally, suppose that $G=C_n$ and that
\[
\kappa_{\LLY}^w(e)\ge0\qquad(e\in E(C_n)).
\]
Then $q_e^w\ge1$ for every edge, while
\[
\sum_{e\in E(C_n)}q_e^w=n
\]
by \cref{lem:edge-set-q}.
Hence \(q_e^w=1\) for every edge. Label the edges of \(C_n\) cyclically by
\[
e_1,\ldots,e_n,
\]
and write \(w_i:=w_{e_i}\), with indices taken modulo \(n\). For each \(i\),
\[
1=q_{e_i}^w
=
\frac{w_i}{w_{i-1}+w_i}
+
\frac{w_i}{w_i+w_{i+1}},
\]
and hence
\[
w_i^2=w_{i-1}w_{i+1}.
\]
Equivalently,
\[
\frac{w_i}{w_{i-1}}
=
\frac{w_{i+1}}{w_i}
\qquad(i\!\!\!\pmod n).
\]
Thus all successive ratios \(w_i/w_{i-1}\) are equal to some \(r>0\). Since
\[
1=\prod_{i=1}^n\frac{w_i}{w_{i-1}}=r^n,
\]
we have \(r=1\). Therefore
\[
w_1=\cdots=w_n,
\]
so \(w\) is constant.
\end{proof}

\section*{Declaration of generative AI use}

During the preparation of this manuscript, the author used ChatGPT
(OpenAI) for brainstorming, exploratory searches for examples and
references, and suggestions concerning exposition and LaTeX
organization. All mathematical statements, proofs, computations, and
bibliographic information were independently checked by the author,
who takes full responsibility for the content of the manuscript.

\end{document}